\documentclass{birkjour_t2}

\usepackage{amsmath,amssymb}
\usepackage{amsthm,amsfonts,latexsym}

\usepackage[mathscr]{euscript}
\usepackage{algorithm,algorithmic}
\usepackage{graphics,graphicx}
\usepackage{enumerate}
\usepackage{cleveref}
\usepackage{color}
\usepackage[sort,numbers]{natbib} 

\makeatletter
\def\weitdacherl#1{\mathop{\vbox{\m@th\ialign{##\crcr\noalign{\kern3\p@}%
      $\hfil{\scriptscriptstyle\frown}\hfil$\crcr\noalign{\kern1\p@\nointerlineskip}%
      $\displaystyle{#1}$\crcr}}}}
\def\weithutzl#1{\mathop{\vbox{\m@th\ialign{##\crcr\noalign{\kern3\p@}%
      $\hfil{\scriptscriptstyle\smile}\hfil$\crcr\noalign{\kern1\p@\nointerlineskip}%
      $\displaystyle{#1}$\crcr}}}}
\def\stern#1{\mathop{\vbox{\m@th\ialign{##\crcr\noalign{\kern3\p@}%
      $\hfil{\scriptscriptstyle *}\hfil$\crcr\noalign{\kern1\p@\nointerlineskip}%
      $\displaystyle{#1}$\crcr}}}}
 
\def\part{\@startsection{part}{0}%
  \z@{2.5\linespacing\@plus\linespacing}{1.0\linespacing}%
  {\LARGE\bfseries\raggedright}}
\makeatother
\newcommand\dirmin{\ensuremath{\weithutzl{\times}}\,}
\newcommand\dirmax{\ensuremath{\weitdacherl{\times}}\,}
\newcommand\dirnon{\ensuremath{\widetilde{\times}}\,}
\newcommand\dir{\ensuremath{\stern{\times}}\,}

\newcommand\strmin{\ensuremath{\weithutzl{\boxtimes}}\,}
\newcommand\strmax{\ensuremath{\weitdacherl{\boxtimes}}\,}

\newtheorem{thm}{Theorem}%[section]

\newtheorem{lemma}[thm]{Lemma}
\newtheorem{prop}[thm]{Proposition}

\theoremstyle{definition}

\theoremstyle{remark}

\numberwithin{equation}{section}

\crefname{thm}{}{}

\begin{document}
\sloppy
%-------------------------------------------------------------------------
% editorial commands: to be inserted by the editorial office
%
%\firstpage{1}
%\volume{228}
%\DOI{003-0001}
%\seriesextraline{This is the Concrete Title of this Book\br H.E. R and S.T.C. W, Eds.}
%\firstpage{1}
\title%[]
 {Associativity and non-associativity of some hypergraph products}
% {On the (non)associativity of some products of hypergraphs}

%----------Author 1
\author{Richard H.\ Hammack}

\address{%
Department of Mathematics and Applied Mathematics\\ 
Virginia Commonwealth University\\ 
Richmond, VA 23284-2014, USA}

\email{rhammack@vcu.edu}

\author{Marc Hellmuth}

\address{%
Department of Mathematics and Computer Science\\
University of Greifswald\\
Walther-Rathenau-Stra{\ss}e 47, D-17487 Greifswald, Germany\\[0.1cm]
Center for Bioinformatics \\
Saarland University \\
Building E 2.1, Room 413,
P.O. Box 15 11 50,
D-66041 Saarbr\"{u}cken,
Germany }

\email{mhellmuth@mailbox.org}

\author{Lydia Ostermeier}
\address{%
Bioinformatics Group, \\ Department of Computer Science and
Interdisciplinary Center for Bioinformatics \\
University of Leipzig,\\
H{\"a}rtelstra{\ss}e 16-18, D-04107 Leipzig, Germany }
\email{glydia@bioinf.uni-leipzig.de}

\author{Peter F.\ Stadler}
\address{
Bioinformatics Group, \\
Department of Computer Science; and
Interdisciplinary Center for Bioinformatics,\\
University of Leipzig,\\
H{\"a}rtelstra{\ss}e 16-18, D-04107 Leipzig, Germany\\[0.1cm]
Max Planck Institute for Mathematics in the Sciences\\
Inselstrasse 22, D-04103 Leipzig, Germany\\[0.1cm]
RNomics Group, Fraunhofer Institut f{\"u}r Zelltherapie und Immunologie,
Deutscher Platz 5e, D-04103 Leipzig, Germany\\[0.1cm]
Department of Theoretical Chemistry,  University of Vienna,
  W{\"a}hringerstra{\ss}e 17, A-1090 Wien, Austria\\[0.1cm]
Santa Fe Institute, 1399 Hyde Park Rd., Santa Fe, NM87501, USA}
\email{studla@bioinf.uni-leipzig.de}

% ----------classification, keywords, date \subjclass{Primary 99Z99;
% Secondary 00A00}
% 
% \keywords{Hypergraph invariants, products, set systems}
% 
% \date{22 Dec 2011}

\dedicatory{}
\begin{abstract}
  Several variants of hypergraph products have been introduced as
  generalizations of the strong and direct products of graphs. Here we show
  that only some of them are associative. In addition to the Cartesian
    product, these are the minimal rank preserving direct product, and the
    normal product.  Counter-examples are given for the strong product as
    well as the non-rank-preserving and the maximal rank preserving direct
    product.
\end{abstract} 

\maketitle
%%%%%%%%%%%%%%%%%%%%%%%%%%%%%%%%%%%%%%%%%%%%%%%%%%%%%%%%%%%%%%%%%%%%%%%%%
\section{Introduction}

Hypergraphs are natural generalizations of undirected graphs in which
``edges'' may consist of more than two vertices. 
Products of hypergraphs have been well-investigated since the 1960s,
see e.g.\ 
\cite{Berge:Hypergraphs,Black15,Bretto09:HyperCartProd,Doerfler79:CoversDirectProd,HL:15,HON-14,
Imrich67:Mengensysteme,Imrich70:SchwachKartProd,KA:12,KA:15,OstHellmStad11:CartProd,
Sonntag89:HamCart,Zhu92:DirectProd}.

The article~\cite{Hellmuth:12c} surveyed the literature on hypergraph
products.  In addition to well-known constructions such as the Cartesian
product and the square product, it also considered various generalizations
of graph products that had rarely been studied, if at all. In particular,
it considered several variants of hypergraph products that generalize the
direct and strong product of \textit{graphs}, namely the direct products,
$\dirmax$ and $\dirnon$, and the strong product $\strmax$. In addition, it
treated the normal product $\strmin$, a generalization of the strong graph
product, and the direct product $\dirmin$, which were introduced by Sonntag
in the 1990's \cite{Sonntag90:NormalProd,Sonntag:thesis}.

Associativity is an important property of product operators. It is
  e.g.\ implicitly assumed in the standard definition of prime factors and
  thus for decompositions of a given hypergraph into prime factors w.r.t.\
  a given product \cite{IMIZ-75, Hammack:2011a}.

The survey \cite{Hellmuth:12c} mistakenly stated that the direct products
$\dirmax$ and $\dirnon$ and the strong product $\strmax$ are associative.
Here we give a simple counterexample for these cases and prove
associativity of the remaining hypergraph products.  This
contribution is an addendum to the results discussed in
\cite{Hellmuth:12c}.

\section{Preliminaries}

We start our brief discussion with the formal definition of hypergraphs,
and the hypergraph products in question.

A \emph{(finite) hypergraph} $H=(V,E)$ consists of a (finite) vertex set $V$ 
and a collection $E$ of non-empty subsets of $V$.
The \emph{rank} of a hypergraph $H=(V,E)$ is $r(H)=\max_{e\in E}|e|$.
A \emph{homomorphism} from $H_1=(V_1, E_1)$ into $H_2=(V_2, E_2)$ is a map
$\phi: V_1\rightarrow V_2$ for which $\phi(e)$ is an edge in $H_2$ whenever $e$
is an edge in $H_1$. 
A bijective homomorphism $\phi$ whose inverse
is also a homomorphism is called an \emph{isomorphism}.
A hypergraph is \emph{simple} if no edge is contained in any other edge
and each edge contains two or more vertices.

In what follows, we consider six hypergraph products $\Box,
\dirmin,\dirmax,\widetilde{\times},\strmin,$ and $\strmax$.  For each of
these, the vertex set of the product is the Cartesian product of the vertex
sets of its factors. To be more precise, given two hypergraphs
$H_1=(V_1,E_1)$ and $H_2=(V_2,E_2)$ and some product $\circledast
\in\{\Box, \dirmin,\dirmax,\widetilde{\times},\strmin,\strmax\}$, then
$V(H_1\circledast H_2) = V(H_1)\times V(H_2)$.  The edge sets of the
various products are defined as follows.

\medskip\noindent\textbf{Cartesian Product} $\boldsymbol{\Box}$:\\
This is an immediate generalization of the standard Cartesian product
of graphs. Its edges are
\begin{align*} 
  E (H_1\Box H_2)=&\big\{\{x\}\times f \mid x\in V(H_1), f\in E (H_2)\big\}\\
  &\cup \big\{e\times \{y\}\mid e\in E (H_1), y\in V(H_2)\big\}.
\end{align*}

\medskip There are several ways to generalize the direct product of graphs
to a product of hypergraphs. Because we want such products to coincide with
the usual direct product when the factors have rank~2 (and are therefore
graphs) it is necessary to impose some \textit{rank restricting} conditions
on the edges. This can be accomplished in different ways and
  leads to different variants of the direct and strong graph products,
  respectively.

\medskip\noindent\textbf{Minimal Rank Preserving Direct Product 
$\boldsymbol{\weithutzl{\times}}$:}\\
Given $e_1\in E_1$ and $e_2\in E_2$, let 
${r}^{-}_{e_1,e_2}=\min\{|e_1|,|e_2|\}$.
The edge set of this product is defined as		
\begin{equation*}
\begin{split}	%\label{eq:Direct_r-uniform}
    E(H_1\weithutzl{\times} H_2):
    =\left\{e\in\binom{e_1\times e_2}{r^{-}_{e_1,e_2}}\mid e_i\in E_i
		\text{\ and\ }  |p_i(e)|=r^{-}_{e_1,e_2},\ i=1,2 \right\}.
\end{split}
\end{equation*}
The edges are thus the subsets $e\subseteq e_1\times e_2$ (with $e_i\in
E_i$) for which both
projections $p_i:e\to e_i$ are injective and at least one is surjective.

\medskip\noindent\textbf{Maximal Rank Preserving Direct Product 
$\boldsymbol{\weitdacherl{\times}}$:}\\
Given $e_1\in E_1$ and $e_2\in E_2$, let 
${r}^{+}_{e_1,e_2}=\max\{|e_1|,|e_2|\}$.  
The edge set of this product is defined as
\begin{equation*}
\begin{split}	%\label{eq:Direct_r-uniform}
  E(H_1\weitdacherl{\times} H_2)
  := \left\{e\in\binom{e_1\times e_2}{r^{+}_{e_1,e_2}}\mid e_i\in E_i
    \text{\ and\ }  p_i(e)=e_i,\ i=1,2 \right\}.
\end{split}
\end{equation*}
The edges are thus the subsets $e\subseteq e_1\times e_2$ (with $e_i\in
E_i$) for which both projections $p_i:e\to e_i$ are surjective and at least
one is injective.

\medskip\noindent\textbf{Non-rank-preserving Direct Product 
$\boldsymbol{\widetilde{\times}}$:}
\begin{align*}
  E(H_1\widetilde{\times} H_2):=
  \big\{\{(x,y)\}\cup \big((e\setminus\{x\})\times(f\setminus\{y\}) \big)\mid 
  x\in e\in  E_1;\ y\in f\in E_2\big\}.	
\end{align*}
\medskip The strong product of graphs is defined as $E(G_1\boxtimes
G_2)=E(G_1\Box G_2)\cup E(G_1\times G_2)$. This leads to the following
generalizations to hypergraphs.

\medskip\noindent\textbf{Normal Product $\boldsymbol{\strmin}$:}
  $$E(H_1\strmin H_2)=E(H_1\Box H_2)\cup E(H_1\dirmin H_2).$$

\medskip\noindent\textbf{Strong Product $\boldsymbol{\strmax}$:}
  $$E(H_1\strmax H_2)=E(H_1\Box H_2)\cup E(H_1\dirmax H_2).$$

\section{Associativity and Non-associativity of Hypergraph Products}

It is well known that the Cartesian product is associative
\cite{Imrich67:Mengensysteme}. In contrast, we will show below that none of
the products $\dirmax$, $\dirnon$ and $\strmax$ is associative. Our
counterexamples require the following lemma.

\begin{lemma} \label{lem:dirmaxdirnon}
If $G$ and $H$ are simple hypergraphs with $r(G)=2$ and $r(H)\leq 3$,
then $G\dirmax H = G\dirnon H$.
\end{lemma}

 \begin{proof}
  By definition, $V(G\dirmax H)=V(G\dirnon H)$.
  We need to show  that $E(G\dirmax H)= E(G\dirnon H)$.
  Given $*\in\{\sim ,\frown\}$ and edges $e_1,e_2$, let $e_1\dir e_2$ 
  denote the set 
  $E\big((e_1,\{e_1\})\dir (e_2,\{e_2\})\big)$.
  Then
  \[E(G\dir H)=\bigcup_{e_1\in E(G), e_2\in E(H)}(e_1\dir e_2).\]
  It suffices to show that $e_1\dirmax e_2=e_1\dirnon e_2$ holds for all 
  $e_1\in E(G)$ and $e_2\in E(H)$.
  Therefore, let $e_1\in E(G)$ and $e_2\in E(H)$ and assume first $|e_2|=2$.
  Say $e_1=\{x_1,y_1\}$ and $e_2=\{x_2,y_2\}$.
  Then $$e_1\dirmax e_2 = 
  \big\{\{(x_1,x_2),(y_1,y_2)\},\{(x_1,y_2),(y_1,x_2)\}\big\}=e_1\dirnon e_2.$$
  Now suppose $|e_2|=3$, say $e_2=\{x_2,y_2,z_2\}$.
  Then 
  \begin{align*}
    e_1\dirmax e_2 = &
    \big\{\{(x_1,x_2),(x_1,y_2),(y_1,z_2)\},\{(x_1,x_2),(y_1,y_2),(y_1,z_2)\},\\
    &\{(x_1,x_2),(y_1,y_2),(x_1,z_2)\},\{(y_1,x_2),(x_1,y_2),(x_1,z_2)\},\\
    &\{(y_1,x_2),(x_1,y_2),(y_1,z_2)\},\{(y_1,x_2),(y_1,y_2),(x_1,z_2)\}\big\}=e_1\dirnon e_2.
  \end{align*}
  Thus the assertion follows. 
\end{proof}

Now we present a counterexample showing that none of the products
$\dirmax$, $\dirnon$ and $\strmax$ is associative.

\bigskip\par\noindent\textbf{Counterexample.}  
Consider the two
hypergraphs  $G=(\{a,b\},\{\{a,b\}\})$ and $H=(\{x,y,z\},\{\{x,y,z\}\})$. 
For $\circledast\in\{\dirmax,\dirnon,\strmax\}$, we claim 
$G\circledast(G\circledast H) \not\cong (G\circledast G)\circledast H$.  
Put
$e=\{(a,(a,x)),(a,(b,y)),(b,(b,z))\}$. Note that $e$ is an edge of
$G\dirmax(G\dirmax H)$, and hence also of $G\strmax(G\strmax H)$.  However, the
set $\{((a,a),x),((a,b),y),((b,b),z)\}$ is not an edge in $(G\strmax
G)\strmax H$, thus also not in $(G\dirmax G)\dirmax H$, because
$\{(a,a),(a,b),(b,b)\}$ is neither an edge in $G\strmax G$ nor in $G\dirmax
G$. Thus the map $(g,(g',h))\mapsto((g,g'),h)$ 
is not an isomorphism $G\dirmax(G\dirmax H)\to (G\dirmax G)\dirmax H$,
nor is it an isomorphism $G\strmax(G\strmax H)\to(G\strmax G)\strmax H$. 
Moreover, the following argument shows there is no isomorphism at all.

It is shown in \cite{HON-14} that the number of edges in $H_1\dirmax
H_2$ is
  \[|E(H_1\dirmax H_2)|=  
  \sum_{e_1\in E_1,e_2\in E_2}(\min\{|e_1|,|e_2|\})!S_{\max\{|e_1|,|e_2|\},\min\{|e_1|,|e_2|\} },\]
  where
  $S_{n,k} = \frac{1}{k!}\sum_{j=0}^k (-1)^{k-j} \binom{k}{j} j^n$
  is a Stirling number of the second kind.
Furthermore,
\begin{eqnarray*}
 |E(H_1\strmax H_2)|&=&|E(H_1\dirmax H_2)|+|E(H_1\Box H_2)|\\
 &=&|E(H_1\dirmax H_2)|+|V(H_1)||E(H_2)|+|E(H_1)||V(H_2)|.
\end{eqnarray*}
Using this, we see that
$|E(G\dirmax(G\dirmax H))|=36\neq 12=|E((G\dirmax G)\dirmax H)|$ and
$|E(G\strmax(G\strmax H))|=82\neq 58=|E((G\strmax G)\strmax H)|$.
Thus $G\dirmax(G\dirmax H)\not\cong(G\dirmax G)\dirmax H$ and
$G\strmax(G\strmax H)\not\cong(G\strmax G)\strmax H$.
%cannot be isomorphic because they do not have the same number of edges.

Moreover, Lemma~\ref{lem:dirmaxdirnon} implies 
$G\dirnon(G\dirnon H)=G\dirmax(G\dirmax H)\neq 
(G\dirmax G)\dirmax H=(G\dirnon G)\dirnon H$.

\bigskip 

The remainder of this contribution proves that the direct product $\dirmin$
and the normal product $\strmin$ are associative. To our knowledge, these
results have not yet appeared in the literature.
\begin{prop}
The direct product $\dirmin$ is associative.
 \label{lem:dirmin}
\end{prop}
\begin{proof}
  Let $H_1=(V_1,E_1),$ $H_2=(V_2,E_2),$ and $H_3=(V_3,E_3)$ be
  hypergraphs and consider the map
  $\psi:V\big(H_1\dirmin (H_2\dirmin H_3)\big)\rightarrow V\big((H_1\dirmin H_2)\dirmin H_3\big)$ defined as
  $(x,(y,z))\mapsto((x,y),z)$. % with $x\in V_1$, $y\in V_2$ and $z\in V_3$.
  We will show that $\psi$ is an isomorphism. Clearly $\psi$ is
  bijective. Hence it remains to show the isomorphism property, that is,
  $e$ is an edge in $H_1\dirmin(H_2\dirmin H_3)$ if and only if $\psi(e)$
  is an edge in $(H_1\dirmin H_2)\dirmin H_3$.  Let
  $e=\{((x_1,y_1),z_1),\ldots,((x_r,y_r),z_r)\}$ be an edge in $(H_1\dirmin
  H_2)\dirmin H_3$.  There are two cases that can occur.

  First, $\{z_1,\ldots,z_r\}$ is an edge in $H_3$ and
  $\{(x_1,y_1),\ldots,(x_r,y_r)\}$ is therefore a subset of an edge in
  $H_1\dirmin H_2$. Hence $\{x_1,\ldots,x_r\}$ and $\{y_1\ldots,y_r\}$ must
  be subsets of edges in $H_1$ and $H_2$, respectively.  But then
  $\{(y_1,z_1),\ldots,(y_r,z_r)\}$ is an edge in $H_2\dirmin H_3$, which implies that
  $\psi (e)=\{(x_1,(y_1,z_1)),\ldots,(x_r,(y_r,z_r))\}$ is an edge in
  $H_1\dirmin(H_2\dirmin H_3)$.  

  Second, $\{(x_1,y_1),\ldots,(x_r,y_r)\}$ is an edge in $H_1\dirmin H_2$
  and $\{z_1,\ldots,z_r\}$ is a subset of an edge in $H_3$.  Then
  $\{x_1,\ldots,x_r\}$ is an edge in $H_1$ and $\{y_1\ldots,y_r\}$ is a
  subset of an edge in $H_2$, or vice versa.  In the first case
  $\{(y_1,z_1),\ldots,(y_r,z_r)\}$ is a subset of an edge in $H_2\dirmin
  H_3$, hence $\psi (e)$ is an edge in $H_1\dirmin (H_2\dirmin H_3)$, and
  in the second case $\{(y_1,z_1),\ldots,(y_r,z_r)\}$ is an edge in
  $H_2\dirmin H_3$ and thus $\psi (e)$ is an edge in 
  $H_1\dirmin(H_2\dirmin H_3)$.  

  This implies that if $e$ is an edge in $(H_1\dirmin H_2)\dirmin H_3$,
  then $\psi(e)$ is an edge in $H_1\dirmin (H_2\dirmin H_3)$. The
  converse
  %, i.e., if $\psi (e)$ is an edge in $H_1\dirmin (H_2\dirmin H_3)$,
  %then $e$ is an edge in $(H_1\dirmin H_2)\dirmin H_3$
  follows analogously. Thus $(H_1\dirmin H_2)\dirmin H_3\cong H_1\dirmin
  (H_2\dirmin H_3)$.% , so $\dirmin$ is associative.
\end{proof}

\begin{prop} %\label{lem:DirectAss}
  The normal product $\strmin$ is associative.
\end{prop}
\begin{proof}
As in the previous proof,
consider the bijection $\psi:V\big((H_1\strmin H_2)\strmin
  H_3\big)\rightarrow V\big(H_1\strmin (H_2\strmin H_3)\big)$ defined as
  $((x,y),z)\mapsto(x,(y,z))$.
We claim this is an isomorphism.

Let $p_{1,2}$ be the projection from
$(H_1\boxtimes H_2)\boxtimes H_3$ onto $H_1\boxtimes H_2$, defined by
$p_{1,2}(((x,y),z))=(x,y)$. Let $p_{2,3}$ be projection 
from $H_1\boxtimes(H_2\boxtimes H_3)$ to $H_2\boxtimes H_3$, 
whereas $p_j$ is the usual projection to $H_j$.
By definition,
$e=\{((x_1,y_1),z_1),\ldots,((x_r,y_r),z_r)\}$ is an edge in $(H_1\strmin
H_2)\strmin H_3$ if and only if one of the following conditions is
satisfied:
\begin{itemize}
\item[(i)]   $p_{1,2}(e)=e_{1,2}\in E(H_1\strmin H_2)$ and $|p_3(e)|=1$,
\item[(ii)]  $p_{3}(e)=e_{3}\in E(H_3)$ and $|p_{1,2}(e)|=1$,
\item[(iii)] $p_{1,2}(e)=e_{1,2}\in E(H_1\strmin H_2)$ and 
  $p_3(e)\subseteq e_3\in E(H_3)$
  and $|e|=|e_{1,2}|=|p_{1,2}(e)|=|p_3(e)|\leq|e_3|$, 
\item[(iv)] $p_3(e)=e_3\in E(H_3)$ and 
  $p_{1,2}(e)\subseteq e_{1,2}\in E(H_1\strmin H_2)$
  and $|e|=|e_{3}|=|p_{3}(e)|=|p_{1,2}(e)|\leq|e_{1,2}|$.
\end{itemize}
   
\medskip\noindent
Condition (i) is equivalent to one of the following conditions holding:
\begin{itemize}
\item[(i\,a)] $p_1(e)=p_1(e_{1,2})=e_1\in E(H_1)$ and
  $|p_2(e)|=|p_2(e_{1,2})|=|p_3(e)|=1$, or
\item[(i\,b)] $p_2(e)=p_2(e_{1,2})=e_2\in E(H_2)$ and
  $|p_1(e)|=|p_1(e_{1,2})|=|p_3(e)|=1$, or
\item[(i\,c)] $p_1(e)=p_1(e_{1,2})=e_1\in E(H_1)$ and
  $p_2(e)=p_2(e_{1,2})\subseteq e_2\in E(H_2)$ and
  $|e|=|e_{1}|=|p_{1}(e)|=|p_2(e)|\leq|e_2|$ and $|p_3(e)|=1$, or
\item[(i\,d)] $p_2(e)=p_2(e_{1,2})=e_2\in E(H_2)$ and
  $p_1(e)=p_1(e_{1,2})\subseteq e_1\in E(H_1)$ and
  $|e|=|e_{2}|=|p_{2}(e)|=|p_1(e)|\leq|e_1|$ and $|p_3(e)|=1$.
\end{itemize}

\medskip\noindent
Condition (iii) is equivalent to one of the following conditions holding:
\begin{itemize}
\item[(iii\,a)] $p_1(e)=p_1(e_{1,2})=e_1\in E(H_1)$ and
  $|p_2(e)|=|p_2(e_{1,2})|=1$ and $p_3(e)\subseteq e_3\in E(H_3)$ and
  $|e|= |p_3(e)|\leq |e_3|$, or
\item[(iii\,b)] $p_2(e)=p_2(e_{1,2})=e_2\in E(H_2)$ and
  $|p_1(e)|=|p_1(e_{1,2})|=1$ and $p_3(e)\subseteq e_3\in E(H_3)$ and
  $|e|= |p_3(e)|\leq |e_3|$, or
\item[(iii\,c)] $p_1(e)=p_1(e_{1,2})=e_1\in E(H_1)$ and
  $p_2(e)=p_2(e_{1,2})\subseteq e_2\in E(H_2)$ and
  $|e|=|e_{1}|=|p_{1}(e)|=|p_2(e)|\leq|e_2|$ and $p_3(e)\subseteq e_3\in
  E(H_3)$ and $|e|= |p_3(e)|\leq |e_3|$, or
\item[(iii\,d)] $p_2(e)=p_2(e_{1,2})=e_2\in E(H_2)$ and
  $p_1(e)=p_1(e_{1,2})\subseteq e_1\in E(H_1)$ and
  $|e|=|e_{2}|=|p_{2}(e)|=|p_1(e)|\leq|e_1|$ and $p_3(e)\subseteq e_3\in
  E(H_3)$ and $|e|= |p_3(e)|\leq |e_3|$.
\end{itemize}

\medskip\noindent
Condition (iv) is equivalent to one of the following conditions holding:
\begin{itemize}
\item[(iv\,a)] $p_3(e)=e_3\in E(H_3)$ and
  $p_1(e)=p_1(p_{1,2}(e))\subseteq e_1\in E(H_1)$ and
  $|p_2(e)|=|p_2(p_{1,2})(e)|=1$ and $|e|=|p_3(e)|=|e_3|=|p_1(e)|\leq
  |e_1|$, or
\item[(iv\,b)] $p_3(e)=e_3\in E(H_3)$ and
  $p_2(e)=p_2(p_{1,2}(e))\subseteq e_2\in E(H_2)$ and
  $|p_1(e)|=|p_1(p_{1,2})(e)|=1$ and $|e|=|p_3(e)|=|e_3|=|p_2(e)|\leq
  |e_2|$, or
\item[(iv\,c)] $p_3(e)=e_3\in E(H_3)$ and
  $p_1(e)=p_1(p_{1,2}(e))\subseteq e_1\in E(H_1)$ and
  $p_2(e)=p_2(p_{1,2}(e))\subseteq e_2\in E(H_2)$ and
  $|e|=|p_3(e)|=|e_3|=|p_1(e)|=|p_2(e)|\leq \min_{i=1,2}|e_i|$.
\end{itemize}

\medskip\noindent
Then Condition (i\,a) implies the following condition:
\begin{itemize}
\item[(I)] $p_1(e)=e_1\in E(H_1)$ and $|p_{2,3}(e)|=1$.
\end{itemize}

\medskip\noindent
Conditions (i\,b), (ii), (iii\,b) and (iv\,b) each imply the following 
condition:
\begin{itemize}
\item[(II)] $|p_1(e)|=1$ and $p_{2,3}(e)=e_{2,3}\in E(H_2\strmin H_3)$.
\end{itemize}

\medskip\noindent
Conditions (i\,c), (iii\,a) and (iii\,c) each imply the following condition:
\begin{itemize}
\item[(III)] $p_1(e)=e_1\in E(H_1)$ and $p_{2,3}(e)\subseteq e_{2,3}\in
  E(H_2\strmin H_3)$ and
  $|e|=|e_{1}|=|p_{1}(e)|=|p_{2,3}(e)|\leq|e_{2,3}|$.
\end{itemize}

\medskip\noindent
Conditions (i\,d), (iii\,d), (iv\,a) and (iv\,c) each imply the following 
condition:
\begin{itemize}
\item[(IV)] $p_1(e)\subseteq e_1\in E(H_1)$ and $p_{2,3}(e)= e_{2,3}\in
  E(H_2\strmin H_3)$ and
  $|e|=|e_{2,3}|=|p_{2,3}(e)|=|p_{1}(e)|\leq|e_{1}|$.
\end{itemize}

\medskip
By definition of the normal product, if any of the Conditions (I)--(IV) 
are satisfied, then
$\psi(e)=\{(x_1,(y_1,z_1)),\ldots,(x_r,(y_r,z_r))\}$ is an
edge in $H_1\strmin(H_2\strmin H_3)$.
It follows that $\psi$ is a homomorphism.
In the same way, the inverse $(x,(y,z))\mapsto ((x,y),z)$ is also a 
homomorphism.
\end{proof}

%\subsection*{Acknowledgements} 
%We thank Richard Hammack for \mh{ providing the counterexample and} 
%drawing our attention to the fact that the
%associativity of hypergraph products is less trivial than we had originally
%thought.

\bibliographystyle{plain} 
\bibliography{HG_final-corrigendum}

\end{document}